\documentclass[
final
]{dmtcs-episciences}

\usepackage[utf8]{inputenc}
\usepackage{subfigure}

\usepackage[round]{natbib}

\usepackage{tikz}
\usetikzlibrary{decorations.pathreplacing}
\usepackage{amssymb}
\usepackage{amsmath}
\usepackage{amsthm}
\theoremstyle{definition}
\newtheorem{definition}{Definition}[section]
\newtheorem{theorem}[definition]{Theorem}
\newtheorem{lemma}[definition]{Lemma}
\newtheorem{proposition}[definition]{Proposition}
\newtheorem{claim}[definition]{Claim}
\newtheorem{corollary}[definition]{Corollary}
\newtheorem{conjecture}[definition]{Conjecture}

\author{Gyula Y.~Katona\affiliationmark{1}\thanks{The research of the first author was supported by National Research, Development and Innovation Office NKFIH, K-116769 and K-124171, by the National Research, Development and Innovation Fund (TUDFO/51757/2019-ITM, Thematic Excellence Program), and by the Higher Education Excellence Program of the Ministry of Human Capacities in the frame of Artificial Intelligence research area of Budapest University of Technology and Economics (BME FIKP-MI/SC).}
\and Kitti Varga\affiliationmark{1,2,3}\thanks{The research of the second author was supported by National Research, Development and Innovation Office NKFIH, K-124171.}}

\title[Minimal toughness in special graph classes]{Minimal toughness in special graph classes}

\affiliation{
  Department of Computer Science and Information Theory, Budapest University of Technology and Economics, Hungary\\
  Alfr\'{e}d R\'{e}nyi Institute of Mathematics, Hungary\\
  ELKH-ELTE Egerv\'{a}ry Research Group, Hungary}

\begin{document}
\publicationdata{vol. 25:3 special issue ICGT'22}{2023}{1}{10.46298/dmtcs.10180}{2022-10-20; 2022-10-20; 2023-06-22}{2023-10-19}

\maketitle
\begin{abstract}
 Let $t$ be a positive real number. A graph is called $t$-tough if the removal of any vertex set $S$ that disconnects the graph leaves at most $|S|/t$ components, and all graphs are considered 0-tough. The toughness of a graph is the largest $t$ for which the graph is $t$-tough, whereby the toughness of complete graphs is defined as infinity. A graph is minimally $t$-tough if the toughness of the graph is $t$, and the deletion of any edge from the graph decreases the toughness. In this paper, we investigate the minimum degree and the recognizability of minimally $t$-tough graphs in the classes of chordal graphs, split graphs, claw-free graphs, and $2K_2$-free graphs.
\end{abstract}

\section{Introduction}

All graphs considered in this paper are finite, simple and undirected. Let $\omega(G)$ denote the \emph{number of components}, $\Delta(G)$ the \emph{maximum degree} and $\kappa(G)$ the \emph{connectivity number} of the graph $G$. (Using $\omega(G)$ to denote the number of components might be confusing; most of the literature on toughness, however, uses this notation.) For a connected graph $G$, a vertex set $S \subseteq V(G)$ is called a \emph{cutset} if its removal disconnects the graph.

The notion of toughness was introduced by~\cite{article:toughness_intro} to investigate Hamiltonicity.

\begin{definition}
 Let $t$ be a real number. A graph $G$ is called \emph{$t$-tough} if $|S| \ge t \cdot \omega(G-S)$ holds for any vertex set $S \subseteq V(G)$ that disconnects the graph (i.e., for any $S \subseteq V(G)$ with $\omega(G-S) > 1$). The \emph{toughness} of $G$, denoted by $\tau(G)$, is the largest $t$ for which G is $t$-tough, taking $\tau(K_n) = \infty$ for all $n \ge 1$.
 
 We say that a cutset $S \subseteq V(G)$ is a \emph{tough set} if $\omega(G - S) = |S|/\tau(G)$.
\end{definition}

Note that a graph is disconnected if and only if its toughness is 0.

Obviously, the more edges a graph has, the larger its toughness can be. The main focus of this work is on graphs whose toughness decreases whenever any of their edges is deleted.

\begin{definition}
 A graph $G$ is said to be \emph{minimally $t$-tough} if $\tau(G) = t$ and $\tau(G - e) < t$ for all $e \in E(G)$.
\end{definition}

For instance, every complete graph on at least two vertices is minimally $\infty$-tough. Note that since the toughness of any noncomplete graph is a rational number, there exist no minimally tough graphs with irrational toughness.

It follows directly from the definition that every $t$-tough, noncomplete graph is $2t$-connected, where $t$ is a nonnegative real number, implying $\kappa(G) \ge 2 \tau(G)$ for noncomplete graphs. Therefore, the minimum degree of any $t$-tough, noncomplete graph is at least $\lceil 2t \rceil$ for any nonnegative real number.

The following conjecture is motivated by a theorem of~\cite{article:min_k_connected}, which states that every minimally $k$-connected graph has a vertex of degree $k$, where $k$ is a positive integer.

\begin{conjecture}[Kriesell,~\cite{article:kriesell_conjecture}] \label{conj:kriesell}
 Every minimally $1$-tough graph has a vertex of degree $2$.
\end{conjecture}

This conjecture can be naturally generalized to any positive real number~$t$.

\begin{conjecture}[Generalized Kriesell conjecture] \label{conj:general_kriesell}
 Every minimally $t$-tough graph has a vertex of degree $\lceil 2t \rceil$, where $t$ is a positive real number.
\end{conjecture}

In the article by~\cite{article:min1tough}, we proved that every minimally $1$-tough graph has a vertex of degree at most $n/3 + 1$. Moreover, we also showed that Conjecture~\ref{conj:kriesell} holds for claw-free graphs by proving that every minimally 1-tough, claw-free graph is a cycle of length at least four. However, in general, determining whether a graph is minimally $t$-tough is hard: in the paper by~\cite{article:dp}, we proved that this problem is DP-complete for any positive rational number $t$. We remark that DP-complete problems are believed to be even harder than NP-complete ones; for more details about the complexity class DP, see the article by~\cite{article:dp_intro}.

Proving Conjecture~\ref{conj:general_kriesell} in general seems to be difficult, so the main motivation of this paper is to study the conjecture in some special graph classes. In the classes of chordal graphs, split graphs, and claw-free graphs, we can prove the conjecture for some values of $t$ by giving a characterization of the minimally $t$-tough graphs in the special graph class, which easily implies the affirmative answer to the conjecture. In the case of $2K_2$-free graphs, the conjecture remains open, but we can at least show that these graphs can be recognized in polynomial time, which may give a chance to find a characterization later. The interesting property of these graph classes is that they are not closed under edge-deletion.

We prove that Conjecture~\ref{conj:general_kriesell} holds
\begin{itemize}
 \item[--] for minimally $t$-tough, chordal graphs when $t \le 1$ is an arbitrary positive rational number,
 \item[--] for minimally $t$-tough, split graphs when $t$ is an arbitrary positive rational number,
 \item[--] for minimally $t$-tough, claw-free graphs when $t \le 1$ is an arbitrary positive rational number.
\end{itemize}

In addition, we show that
\begin{itemize}
 \item[--] minimally $t$-tough, split graphs for any positive rational number $t$, and
 \item[--] minimally $t$-tough, claw-free graphs for any positive rational number $t \le 1$, and
 \item[--] minimally $t$-tough, $2K_2$-free graphs for any positive rational number $t$
\end{itemize}
can be recognized in polynomial time.

Moreover, we give a characterization
\begin{itemize}
 \item[--] of minimally $t$-tough, chordal graphs when $t \in (1/2, 1]$ is an arbitrary rational number (more precisely, we show that there exist no such graphs),
 \item[--] of minimally $t$-tough, split graphs when $t$ is an arbitrary positive rational number (more precisely, we show that there exist such graphs only when $t$ is the reciprocal of an integer $b \ge 2$),
 \item[--] of minimally $t$-tough, claw-free graphs when $t \le 1$ is an arbitrary positive rational number (by a theorem of Matthews and Sumner, we know that the toughness of a noncomplete, claw-free graph is either an integer or a half-integer, so we only need to study the cases when $t=1/2$ or $t=1$, and the latter case was already handled in our article~\cite{article:min1tough}).
\end{itemize}

\section{Preliminaries}

In this section, we study some basic properties of minimally tough graphs.

The following proposition is a simple observation.

\begin{proposition} \label{prop:obs_below_1}
 Let $t \le 1$ be a positive rational number and $G$ a graph with $\tau(G) = t$. Then 
 \[ \omega(G-S) \le |S|/t \]
 for any nonempty proper subset $S$ of $V(G)$.
\end{proposition}
\begin{proof}
 If $S$ is a cutset in $G$, then by the definition of toughness, $\omega(G-S) \le |S|/t$ holds.

 If $S$ is not a cutset in $G$, then $\omega(G-S) = 1$ since $S \ne V(G)$. On the other hand, $|S|/t \ge 1$ since $S \ne \emptyset$ and $t \le 1$. Therefore, $\omega(G-S) \le |S|/t$ holds in this case as well.
\end{proof}

As is clear from its proof, the above proposition holds even if $S$ is not a cutset. However, it does not necessarily hold if $t > 1$ and $S$ is not a cutset: if $t>1$, then the graph cannot contain a cut-vertex, therefore $\omega(G-S) = 1$ for any subset $S$ with $|S| = 1$, while $|S|/t = 1/t < 1$.

The following proposition describes the basic structure of minimally tough graphs.

\begin{proposition} \label{prop:minttoughlemma}
 Let $t$ be a positive rational number and $G$ a minimally $t$-tough graph. For every edge $e$ of $G$,
 \begin{enumerate}
  \item[--] the edge $e$ is a bridge in $G$, or
  \item[--] there exists a vertex set ${S=S(e) \subseteq V(G)}$  with 
 \[ \omega(G-S) \le \frac{|S|}{t} \quad \mathrm{and} \quad \omega \big( (G-e)-S \big) > \frac{|S|}{t}, 
 \]
 and the edge $e$ is a bridge in $G-S$.
 \end{enumerate}
 In the first case, we define $S = S(e) = \emptyset$.
\end{proposition}
\begin{proof}
 Let $e$ be an arbitrary edge of $G$ which is not a bridge. Since $G$ is minimally $t$-tough, ${\tau (G-e) < t}$. Since $e$ is not a bridge, $G-e$ is still connected, so there exists a cutset $S=S(e) \subseteq V(G-e)=V(G)$ in $G-e$ satisfying $\omega \big( (G-e)-S \big) > |S|/t$.
 
 By Proposition \ref{prop:obs_below_1}, if $t \le 1$, then $\omega (G-S) \le |S|/t$. So assume that $t > 1$ holds. Then there are two cases.
 
 \medskip
 
 \emph{Case 1:} ($t > 1$ and) $S$ is a cutset in $G$.
 
 Since $\tau (G) = t$ and $S$ is a cutset, $\omega (G-S) \le |S|/t$. This is only possible if $e$ connects two components of $(G-e)-S$, i.e. if $e$ is a bridge in $G-S$.
 
 \medskip
 
 \emph{Case 2:} ($t > 1$ and) $S$ is not a cutset in $G$. 
 
 Then $\omega(G-S) = 1$. Since $S$ is a cutset in $G-e$, the edge $e$ must connect two components of $(G-e)-S$. Thus $e$ is a bridge in $G-S$ and $\omega \big( (G-e) - S \big) = 2$.
 
 Now we show that $\omega(G-S) \le |S|/t$ holds. Suppose to the contrary that $\omega(G-S) > |S|/t$. Since $\omega(G-S) = 1$, this implies $|S| < t$. Moreover, since $\tau(G) = t$, the graph $G$ is $\lceil 2t \rceil$-connected, thus it has at least $2t+1$ vertices. From this, it follows that $S$ and one of the endpoints of $e$ form a cutset in~$G$: otherwise, $G$ would only have $|S|+2 < t+2 < 2t+1$ vertices (where the latter inequality is valid since $t > 1$). Let $S'$ denote this cutset. Since $G$ is $t$-tough and $S'$ is a cutset in $G$, we obtain
 \[ 2 \le \omega(G-S') \le \frac{|S'|}{t} = \frac{|S|+1}{t}, \]
 so $|S| \ge 2t-1$. Therefore,
 \[ 2t-1 \le |S| < t, \]
 which implies $t < 1$, and that is a contradiction.

 \medskip
 
 So in both cases,
 \[ \omega(G-S) \le \frac{|S|}{t} \quad \mathrm{and} \quad \omega \big( (G-e)-S \big) > \frac{|S|}{t} \]
 hold, and $e$ is a bridge in $G-S$.
\end{proof}

The starting point of each of our proofs is showing that for each edge~$e$ of the investigated graphs, there exists a vertex set $S(e)$ guaranteed by Proposition~\ref{prop:minttoughlemma} with some nice properties.

\section{Chordal graphs}

First, we study chordal graphs. Unlike for the other graph classes considered in this paper, the complexity of determining the toughness of chordal graphs is open.

\begin{definition}
 A graph is \emph{chordal} if it does not contain an induced cycle of length at least $4$.
\end{definition}

\begin{definition}
 A vertex $v$ of a graph $G$ is \emph{simplicial} if its neighborhood $N(v)$ forms a clique in $G$.
\end{definition}

A classical theorem of~\cite{article:rigid_circuit} states that every chordal graph has a simplicial vertex; moreover, every noncomplete, chordal graph contains at least two nonadjacent simplicial vertices.

The main idea of the following proof is that in a minimally $t$-tough, chordal graph with $t \le 1$, if $e$ is an edge in the neighborhood of a simplicial vertex $v$, then the vertex set $S(e)$ guaranteed by Proposition~\ref{prop:minttoughlemma} should, but cannot, contain $v$.

\begin{theorem} \label{thm:chordal}
 Let $t \le 1$ be a positive rational number. If $t \le 1/2$, then every simplicial vertex of any minimally $t$-tough, chordal graph has degree~1. If $1/2 < t \le 1$, then there exist no minimally $t$-tough, chordal graphs.
\end{theorem}
\begin{proof}
 If $t \le 1/2$, then let $G$ be a minimally $t$-tough, chordal graph, and if $1/2 < t \le 1$, then suppose to the contrary that $G$ is a minimally $t$-tough, chordal graph.

 In both cases, let $v$ be a simplicial vertex of $G$. If $t \le 1/2$, then suppose to the contrary that $v$ has degree at least 2. If $1/2 < t \le 1$, then every vertex of $G$ has degree at least $2t > 1$, so $v$ has degree at least 2.
 
 In both cases, let $u$ and $w$ be two neighbors of $v$. Since $v$ is simplicial, $u$ and $w$ are adjacent. Let $e = uw$. Obviously, $e$ is not a bridge, so by Proposition~\ref{prop:minttoughlemma}, there exists a vertex set $S = S(e) \subseteq V(G)$ such that
 \[ \omega(G-S) \le \frac{|S|}{t} \qquad \mathrm{and} \qquad \omega \big( (G-e) - S \big) > \frac{|S|}{t}, \]
 and $e$ is a bridge in $G-S$. Clearly, $S$ must contain $v$.
 
 \medskip
 
 \emph{Case 1:} $|S| \le 1$, i.e.\ $S = \{ v \}$.
 
 Since $v$ is simplicial, $(G-e) - \{v\}$ has exactly two components, and $v$ has exactly one neighbor in both of them, so the degree of $v$ is exactly $2$.
 
 If $t \le 1/2$, then
  \[ 2 = \omega \big( (G-e) - S \big) > \frac{|S|}{t} \ge \frac{1}{1/2} = 2, \]
 which is a contradiction.
 
 If $1/2 < t \le 1$, then $G - \{ u \}$ or $G - \{ w \}$ is not connected (since $G \not\simeq K_3$), which contradicts the fact that $\kappa(G) \ge 2t > 1$.
 
 \medskip
 
 \emph{Case 2:} $|S| \ge 2$.
 
 Since $v$ is simplicial, $\omega(G-S) = \omega \big( G-(S \setminus \{v\}) \big)$ holds. Since $\tau(G) = t \le 1$ and $S \setminus \{v\} \ne \emptyset$, Proposition \ref{prop:obs_below_1} implies
 \[ \omega \big( G - (S \setminus \{ v \}) \big) \le \frac{\big| S \setminus \{ v \} \big|}{t} = \frac{|S| - 1}{t}. \]
 Hence,
 \[ \frac{|S|}{t} < \omega \big( (G-e) - S \big) = \omega(G-S) + 1 = \omega \big( G- (S \setminus \{v\}) \big) + 1 \le \frac{|S| - 1}{t} + 1 = \frac{|S|}{t} + \left( 1 - \frac{1}{t} \right) \le \frac{|S|}{t}, \]
 where the inequality $1-1/t \le 0$ is valid since $t \le 1$, thus we obtained a contradiction.
\end{proof}

Thus, Conjecture~\ref{conj:general_kriesell} is true for minimally $t$-tough, chordal graphs with $t \le 1$.

Now it is natural to ask if there exist minimally $t$-tough, chordal graphs with $t \le 1/2$. \cite{article:chordal_mintough} proved that the answer is affirmative if and only if $t$ is a reciprocal of an integer $b \ge 2$. Moreover, they gave a complete characterization of minimally $t$-tough, chordal graphs with $t \le 1/2$.

\begin{theorem}[\cite{article:chordal_mintough}]
 A chordal graph is minimally $t$-tough for some positive rational number $t \le 1/2$ if and only if it can be obtained from a tree with maximum degree $\Delta \ge 3$ by removing an independent set $W$ of vertices of degree 3 and connecting the three neighbors of each vertex of $W$ by a triangle, where $W$ satisfies the following.
 \begin{itemize}
  \item[--] Either $\Delta = 3$ and $W$ is the set of all vertices of degree 3 such that every neighbor of any vertex in $W$ has degree 2,
  \item[--] or $\Delta \ge 3$ and $W$ is some subset of vertices of degree 3 such that every neighbor of any vertex in $W$ has degree $\Delta \ge 3$.
 \end{itemize}
\end{theorem}

The above theorem implies that for any positive rational number $t \le 1/2$, the class of minimally $t$-tough, chordal graphs can be recognized in polynomial time.

\section{Split graphs}

In this section, we study split graphs. It is not difficult to see that every split graph is chordal, thus the results of the previous section can be applied here.

\begin{definition}
 A graph is a \emph{split} graph if its vertex set can be partitioned into a clique and an independent set.
\end{definition}

The toughness of split graphs can be computed in polynomial time; first, this was shown for $t=1$ by~\cite{article:split_1}, then by~\cite{article:split_general} for all positive rational numbers~$t$.

\begin{theorem}[\cite{article:split_general}] \label{thm:split_toughness}
 For any positive rational number $t$, the class of $t$-tough, split graphs can be recognized in polynomial time.
\end{theorem}

Note that if an edge goes between the clique and the independent set of the split graph, then after the removal of this edge, the graph is still a split graph, so we can compute in polynomial time whether the toughness decreased. The following lemma says that the vertex sets showing that the removal of any edge in the clique decreases the toughness can be uniquely determined.

\begin{lemma} \label{lemma:split_mintoughset}
 Let $t$ be a positive rational number and $G$ a minimally $t$-tough, split graph partitioned into a clique $Q$ and an independent set $I$. Let $e = uv$ be an edge between two vertices of $Q$ and $S = S(e) \subseteq V(G)$ a vertex set guaranteed by Proposition~\ref{prop:minttoughlemma}. Then 
 \[ S = \big( Q \setminus \{ u, v \} \big) \cup \big\{ w \in I \bigm| uw, \, vw \in E(G) \big\}. \]
\end{lemma}
\begin{proof}
 Let $S = S(e) \subseteq V(G)$ be an arbitrary vertex set guaranteed by Proposition~\ref{prop:minttoughlemma} (such a vertex set exists since $G$ is minimally $t$-tough), and let $S_0 = \big( Q \setminus \{ u, v \} \big) \cup \big\{ w \in I \bigm| uw, \, vw \in E(G) \}$. Now we show that $S = S_0$ holds.
 
 By Proposition~\ref{prop:minttoughlemma}, the edge $e$ must be a bridge in $G-S$, so $u, v \notin S$ and $S_0 \subseteq S$.
 
 Suppose to the contrary that $S_0 \subsetneqq S$. Note that in $(G-e)-S_0$, every vertex except $u$ and $v$ has degree at most 1, so since $u, v \notin S$, we obtain
 \[ \omega(G-S) = \omega(G-S_0) - |S \setminus S_0| \le \omega(G-S_0) - 1, \]
 where the last inequality holds by the assumption $S_0 \subsetneqq S$. Then by Proposition~\ref{prop:minttoughlemma},
 \[ \frac{|S_0|}{t} < \frac{|S|}{t} < \omega \big( (G-e)-S \big) = \omega(G-S) + 1 \le \omega(G-S_0) \le \frac{|S_0|}{t}, \]
 which is a contradiction.
\end{proof}
 
The next theorem states that except for the complete graphs, there exist no minimally $t$-tough, split graphs with $t > 1/2$. The main idea of the proof is to show that if such a graph existed with a partitioning $Q \cup I$ of its vertex set, where $Q$ spans a clique and $I$ an independent set, then every vertex of $Q$ would have at most one neighbor in $I$, but these split graphs cannot be minimally $t$-tough.

\begin{theorem}
 For any rational number $t > 1/2$, there exist no minimally $t$-tough, split graphs.
\end{theorem}
\begin{proof}
 Since every split graph is chordal, we can assume by Theorem~\ref{thm:chordal} that $t > 1$ holds. Suppose to the contrary that there exists a minimally $t$-tough, split graph $G$. Obviously, $G$ is noncomplete, otherwise its toughness would be infinity and not $t$. Let $G$ be partitioned into a clique $Q$ and an independent set $I$. Since $G$ is noncomplete, $I \ne \emptyset$ and since $G$ is connected and noncomplete, $Q \ne \emptyset$ also holds.
 
 Now we show that we can assume $|I| \ge 2$. We already know that $I \ne \emptyset$, so let us consider the case when $|I| = 1$ holds. Since $G$ is noncomplete, we can pick a vertex $w \in Q$ which is not adjacent to the vertex of $I$. Then we can consider $I' = I \cup \{ w \}$ and $Q' = Q \setminus \{ w \}$, instead of $I$ and $Q$.
 
 Note that $|Q| \ge 2$, otherwise $G \simeq K_{1,b}$ for some $b \ge 2$, and so $\tau(G) = 1/b \le 1/2$, which would be a contradiction. So let $e=uv$ be an edge in $Q$ and $S = S(e) \subseteq V(G)$ a vertex set guaranteed by Proposition~\ref{prop:minttoughlemma}. By Lemma~\ref{lemma:split_mintoughset},
 \[ S = \big( Q \setminus \{ u, v \} \big) \cup \{ w \in I \mid uw,vw \in E(G) \}. \]
 Then every component of $G-S$ has size $1$ except for the component of the edge $e$. Let
 \begin{align*}
  x & = \big| \{ w \in I \mid uw,vw \in E(G) \} \big|, \\
  \ell_u & = \big| \{ w \in I \mid uw \in E(G), ~ vw \not\in E(G) \} \big|, \\
  \ell_v & = \big| \{ w \in I \mid uw \not\in E(G), ~ vw \in E(G) \} \big|,
 \end{align*}
 see Figure~\ref{fig:split_x_l}.
 
 \begin{figure}[ht] 
 \begin{center}
 \begin{tikzpicture}
  \tikzstyle{vertex}=[draw,circle,fill=black,minimum size=4,inner sep=0]
  
  \draw[thick](0,0) ellipse (0.6 and 1.5);
  
  \node[vertex] (v1) at (0,0.9) {};
  \node[vertex] (v2) at (0,0.3) {};
  
  \node[vertex] (u1) at (1.5,1.8) {};
  \node[vertex] (u2) at (1.5,1.4) {};
  
  \node[vertex] (u3) at (1.5,0.8) {};
  \node[vertex] (u4) at (1.5,0.4) {};
  
  \node[vertex] (u5) at (1.5,-0.2) {};
  \node[vertex] (u6) at (1.5,-0.6) {};
  
  \node[vertex] at (0,-0.2) {};
  \node[vertex] at (0,-0.6) {};
  \node[vertex] at (0,-1) {};
  
  \node[vertex] at (1.5,-1.2) {};
  \node[vertex] at (1.5,-1.6) {};
  
  \node at (-0.5,1.75) {$Q$};
  \node at (1.25,2.25) {$I$};
  \node at (-0.3,0.9) {$u$};
  \node at (-0.3,0.3) {$v$};
  
  \draw[thick] (v1) -- (v2) node[pos=0.5, right, xshift=-2pt] {$e$};
  \draw[thick] (v1) -- (u1);
  \draw[thick] (v1) -- (u2);
  \draw[thick] (v1) -- (u3);
  \draw[thick] (v1) -- (u4);
  \draw[thick] (v2) -- (u3);
  \draw[thick] (v2) -- (u4);
  \draw[thick] (v2) -- (u5);
  \draw[thick] (v2) -- (u6);
  
  \draw[thick] [decorate,decoration={brace,amplitude=3pt}] (1.85,1.9) -- (1.85,1.3) node [pos=0.5, xshift=12pt] {$\ell_u$};
  \draw[thick] [decorate,decoration={brace,amplitude=3pt}] (1.85,0.9) -- (1.85,0.3) node [pos=0.5, xshift=12pt] {$x$};
  \draw[thick] [decorate,decoration={brace,amplitude=3pt}] (1.85,-0.1) -- (1.85,-0.7) node [pos=0.5, xshift=12pt] {$\ell_v$};
  
  \draw[thick] (-0.75,-1.75) rectangle (0.75,0.1);
  \draw[thick] (1.35,0.25) rectangle (1.65,0.95);
  
  \node at (-1,-1.5) {$S$};
 \end{tikzpicture}
 \caption{The set $S = S(e)$ consisting of the vertices of $Q \setminus \{ u, v \}$ and the common neighbors of $u$ and $v$.}
 \label{fig:split_x_l}
 \end{center}
 \end{figure}
 
 Since $|I| \ge 2$ holds, $Q$ is clearly a cutset in $G$, so
 \[ x + \big( \omega(G-S) - 1 \big) \le |I| = \omega(G - Q) \le \frac{|Q|}{t}. \]
 By Proposition~\ref{prop:minttoughlemma},
 \[ \omega(G-S) \le \frac{|S|}{t} = \frac{|Q| - 2 + x}{t} < \omega \big( (G-e) - S \big) = \omega(G-S) + 1. \]
 Thus
 \[ x-2 + \frac{|Q| - 2 + x}{t} < x-2 + \big( \omega(G-S) + 1 \big) = x + \big( \omega(G-S) - 1 \big) \le \frac{|Q|}{t}, \]
 i.e.,
 \[ x-2 + \frac{x-2}{t} < 0. \]
 Since $t$ is positive and $x$ is an integer, this implies $x \le 1$.
 
 Now we show $l_u \le 1$. Since $S$ is a cutset in $G$ and $u$ is not an isolated vertex in $G-S$, it follows that $S \cup \{ u \}$ is also a cutset in $G$. Thus,
 \[ \omega(G-S) + l_u = \omega \big( G - (S \cup \{ u \}) \le \frac{|S \cup \{ u \}|}{t} = \frac{|S|+1}{t} \]
 holds since $G$ is $t$-tough. On the other hand, by Proposition~\ref{prop:minttoughlemma},
 \[ \omega(G-S) + l_u = \omega \big( (G-e) - S \big) - 1 + l_u > \frac{|S|}{t} + l_u - 1. \]
 Therefore,
 \[ \frac{|S|}{t} + l_u - 1 < \frac{|S|+1}{t}, \]
 i.e., $l_u < 1 + 1/t < 2$, which implies $l_u \le 1$.
 
 Similarly, $l_v \le 1$.
 
 Now we show that $u$ has at most one neighbor in $I$, i.e., $l_u + x \le 1$. Suppose to the contrary that $u$ has more neighbors in $I$. By the above observations, this is only possible if $x=1$ and $l_u = 1$. Then the vertex set $S' = (S \setminus I) \cup \{ u \}$ is a cutset in $G$ and $\omega(G-S') = \omega(G-S) + 1$. In addition, since $G$ is $t$-tough, $\omega(G-S') \le |S'|/t$ holds. Then by Proposition~\ref{prop:minttoughlemma},
 \[ \frac{|S|}{t} < \omega \big( (G-e) - S \big) = \omega(G-S) + 1 = \omega(G-S') \le \frac{|S'|}{t} = \frac{|S|}{t}, \]
 which is a contradiction.
 
 Since $e$ is an arbitrary edge in $Q$, it follows that every vertex of $Q$ has at most one neighbor in $I$. By a previous assumption, $|I| \ge 2$, and note that $G$ is connected (otherwise its toughness would be 0). Therefore, there exists an edge $f \in E(Q)$ both of whose endpoints have a neighbor in $I$ but not a common one. Then by Lemma~\ref{lemma:split_mintoughset},
 \[ \big| S(f) \big| = |Q| - 2 \]
 holds, so clearly 
 \[ \omega \big( G - S(f) \big) = |I|-1, \qquad \omega \big( (G - f) - S(f) \big) = |I|, \]
 and by Proposition~\ref{prop:minttoughlemma},
 \[ |I|-1 \le \frac{|Q| - 2}{t} < |I|. \]
 
 On the other hand, let $a \in I$ be fixed, and let
 \[ R = R(a) = \{ v \in V(Q) \mid av \not\in E(G) \}, \]
 see Figure~\ref{fig:split_R}.
 
 \begin{figure}[ht] 
 \begin{center}
 \begin{tikzpicture}
  \tikzstyle{vertex}=[draw,circle,fill=black,minimum size=4,inner sep=0]
  
  \draw[thick](0,0) ellipse (0.6 and 2);
  
  \node[vertex] (v1) at (0,1.75) {};
  \node[vertex] (v2) at (0,1.25) {};
  \node[vertex] (v3) at (0,0.75) {};
  \node[vertex] (v4) at (0,0.25) {};
  \node[vertex] (v5) at (0,-0.25) {};
  \node[vertex] (v6) at (0,-0.75) {};
  \node[vertex] (v7) at (0,-1.25) {};
  \node[vertex] (v8) at (0,-1.75) {};
  
  \node[vertex] (u1) at (1.5,1.25) {};
  \node[vertex] (u2) at (1.5,0) {};
  \node[vertex] (u3) at (1.5,-1) {};
  
  \node at (-0.25,2.5) {$Q$};
  \node at (1.25,2) {$I$};
  
  \node at (1.85,1.25) {$a$};
  
  \draw[thick] (u1) -- (v1);
  \draw[thick] (u1) -- (v2);
  \draw[thick] (u1) -- (v3);
  \draw[thick] (u2) -- (v4);
  \draw[thick] (u2) -- (v5);
  \draw[thick] (u3) -- (v6);
  \draw[thick] (u3) -- (v7);
  
  \draw[thick] (-0.75,-2.25) rectangle (0.75,0.5);
  
  \node at (-1.1,-2) {$R$};
 \end{tikzpicture}
 \caption{The set $R = R(a)$ consisting of those vertices of $Q$ that are not adjacent to $a$.}
 \label{fig:split_R}
 \end{center}
 \end{figure}
 
 Since $\kappa(G) \ge 2t$, the vertex $a$ has degree at least $2t$, so $|R| \le |Q| - 2t$. Since $|I| \ge 2$, the vertex set $R$ is a cutset in $G$, so
 \[ |I| = \omega(G - R) \le \frac{|R|}{t} \le \frac{|Q| - 2t}{t}. \]
 Hence,
 \[ \frac{|Q| - 2}{t} < |I| \le \frac{|Q| - 2t}{t}, \]
 which is a contradiction since $t > 1$.
\end{proof}

We have just shown that for any rational number $t > 1/2$, there exist no minimally $t$-tough, split graphs. However, note that complete graphs are minimally $\infty$-tough, split graphs. Now we characterize minimally $t$-tough, split graphs for $t \le 1/2$: in the following theorem, we show that $t$ must be a reciprocal of an integer $b \ge 2$, and the class of minimally $1/b$-tough, split graphs can be seen in Figure~\ref{fig:split_characterization}.

\begin{figure}[ht]
\begin{center}
\begin{tikzpicture}
 \tikzstyle{vertex}=[draw,circle,fill=black,minimum size=4,inner sep=0]
 
 \node[vertex] (1) at (0,0) {};
 \node[vertex] (2) at (1,1) {};
 \node[vertex] (3) at (1,0.5) {};
 \node[vertex] (4) at (1,0) {};
 \draw[fill] (1,-0.4) circle (0.3pt);
 \draw[fill] (1,-0.5) circle (0.3pt);
 \draw[fill] (1,-0.6) circle (0.3pt);
 \node[vertex] (5) at (1,-1) {};
 
 \draw[thick] (1) -- (2);
 \draw[thick] (1) -- (3);
 \draw[thick] (1) -- (4);
 \draw[thick] (1) -- (5);
 
 \draw[thick][decorate,decoration={brace,amplitude=5pt}] (1.3,1.1) -- (1.3,-1.1) node [pos=0.5, xshift=12pt] {$b$};
 
 \begin{scope}[shift={(4.25,0)}]
 \node[vertex] (1) at (0,0.5) {};
 \node[vertex] (2) at (0,-0.5) {};
 \node[vertex] (3) at (1,1.4) {};
 \node[vertex] (4) at (1,1) {};
 \draw[fill] (1,0.725) circle (0.3pt);
 \draw[fill] (1,0.625) circle (0.3pt);
 \draw[fill] (1,0.525) circle (0.3pt);
 \node[vertex] (5) at (1,0.25) {};
 \node[vertex] (6) at (1,-0.4) {};
 \draw[fill] (1,-0.8) circle (0.3pt);
 \draw[fill] (1,-0.9) circle (0.3pt);
 \draw[fill] (1,-1) circle (0.3pt);
 \node[vertex] (7) at (1,-1.4) {};
 
 \draw[thick] (1) -- (2);
 \draw[thick] (1) -- (3);
 \draw[thick] (1) -- (4);
 \draw[thick] (1) -- (5);
 \draw[thick] (2) -- (6);
 \draw[thick] (2) -- (7);
 
 \draw[thick] [decorate,decoration={brace,amplitude=5pt}] (1.3,1.5) -- (1.3,0.15) node [pos=0.5, xshift=22pt] {$b-1$};
 \draw[thick] [decorate,decoration={brace,amplitude=5pt}] (1.3,-0.3) -- (1.3,-1.5) node [pos=0.5, xshift=29pt] {$\le b-1$};
 \end{scope}
 
 \begin{scope}[shift={(9,0)}]
 \node[vertex] (1) at (0,1) {};
 \node[vertex] (2) at (0,0) {};
 \node[vertex] (3) at (0,-1) {};
 \node[vertex] (4) at (1,1.3) {};
 \draw[fill] (1,1.1) circle (0.3pt);
 \draw[fill] (1,1) circle (0.3pt);
 \draw[fill] (1,0.9) circle (0.3pt);
 \node[vertex] (5) at (1,0.7) {};
 \node[vertex] (6) at (1,0.3) {};
 \draw[fill] (1,0.1) circle (0.3pt);
 \draw[fill] (1,0) circle (0.3pt);
 \draw[fill] (1,-0.1) circle (0.3pt);
 \node[vertex] (7) at (1,-0.3) {};
 \node[vertex] (8) at (1,-0.7) {};
 \draw[fill] (1,-0.9) circle (0.3pt);
 \draw[fill] (1,-1) circle (0.3pt);
 \draw[fill] (1,-1.1) circle (0.3pt);
 \node[vertex] (9) at (1,-1.3) {};
 
 \draw[thick] (1) -- (2) -- (3);
 \draw[thick] (1) to [bend right=45] (3);
 \draw[thick] (1) -- (4);
 \draw[thick] (1) -- (5);
 \draw[thick] (2) -- (6);
 \draw[thick] (2) -- (7);
 \draw[thick] (3) -- (8);
 \draw[thick] (3) -- (9);
 
 \draw[thick] [decorate,decoration={brace,amplitude=5pt}] (1.25,1.4) -- (1.25,0.6) node [pos=0.5, xshift=22pt] {$b-1$};
 \draw[thick] [decorate,decoration={brace,amplitude=5pt}] (1.25,0.4) -- (1.25,-0.4) node [pos=0.5, xshift=22pt] {$b-1$};
 \draw[thick] [decorate,decoration={brace,amplitude=5pt}] (1.25,-0.6) -- (1.25,-1.4) node [pos=0.5, xshift=22pt] {$b-1$};
 \end{scope}
\end{tikzpicture}
\caption{The class of minimally $1/b$-tough, split graphs.}
\label{fig:split_characterization}
\end{center}
\end{figure}

\begin{theorem}
 Let $t \le 1/2$ be an arbitrary positive rational number and $G$ a minimally $t$-tough, split graph partitioned into a clique $Q$ and an independent set $I$. Then there exists an integer $b \ge 2$ for which $t = 1/b$, and $|Q| \le 3$. Moreover, 
 \begin{enumerate}
  \item either $G$ is a tree with at most two non-leaf vertices and with $\Delta(G) = b$,
  \item or $|Q| = 3$, every vertex in $I$ has degree $1$ and every vertex in $Q$ has degree $b+1$.
 \end{enumerate}
\end{theorem}
\begin{proof} 
 Let $G$ be a minimally $t$-tough, split graph with $t \le 1/2$. If $G$ is triangle-free, then it must be a tree, so $\tau(G) = 1/\Delta(G)$. Since $G$ is a split graph, it can have at most two non-leaf vertices. So $G$ is either a star or a double-star (i.e.\ two stars connected by their center vertices).
 
 Let us assume that there is a triangle in $G$. Since every split graph is chordal and every vertex in $I$ is simplicial, Theorem~\ref{thm:chordal} implies that every vertex in $I$ has degree $1$. It is not difficult to see that if such a graph is minimally $t$-tough with $t \le 1/2$, then $|Q| \le 3$. But since there is a triangle in the graph, $|Q| = 3$. Obviously, all vertices of $Q$ must have degree $b+1$.
\end{proof}

This theorem directly implies the following.

\begin{corollary}
 For any positive rational number $t$, the class of minimally $t$-tough, split graphs can be recognized in polynomial time.
\end{corollary}

\begin{corollary}
Let $t$ be a positive rational number. If $G$ is a minimally $t$-tough, split graph, then $G$ has a vertex of degree $\lceil 2t \rceil$.
\end{corollary}

So Conjecture~\ref{conj:general_kriesell} holds for split graphs.

\section{Claw-free graphs}

In this section, we study claw-free graphs.

\begin{definition}
 The graph $K_{1,3}$ is called a {\em claw}. A graph is said to be {\em claw-free} if it does not contain a claw as an induced subgraph.
\end{definition}

By the following theorem, the toughness of claw-free graphs can be computed in polynomial time.

\begin{theorem}[\cite{article:clawfree}] \label{thm:toughness_of_clawfree}
 If $G$ is a noncomplete, claw-free graph, then $2\tau(G) = \kappa(G)$.
\end{theorem}

\begin{corollary}
 For any rational number $t > 0$, the class of $t$-tough, claw-free graphs can be recognized in polynomial time.
\end{corollary}

The following lemma follows directly from the proof of Theorem~\ref{thm:toughness_of_clawfree}, which can be found as Theorem $10$ in the article by~\cite{article:clawfree}.

\begin{lemma} \label{lemma:structure_of_clawfree_graphs_with_tougness_t}
 Let $G$ be a noncomplete, claw-free graph and $S$ a tough set. Then the vertices of $S$ have neighbors in exactly two components of $G-S$, and the components of $G-S$ have exactly $2 \tau(G)$ neighbors (in $S$).
\end{lemma}

Here, we study minimally $t$-tough, claw-free graphs, where $t \le 1$ is a positive rational number. By Theorem~\ref{thm:toughness_of_clawfree}, we only need to consider the cases $t = 1/2$ and $t=1$. The latter case was already settled in our article~\cite{article:min1tough}.

\begin{theorem}[\cite{article:min1tough}] \label{thm:characterization_of_min_1_tough_clawfree}
 The class of minimally $1$-tough, claw-free graphs consists of the cycles of length at least $4$.
\end{theorem}

The main idea of this proof is that for each edge, there exists a vertex set guaranteed by Proposition~\ref{prop:minttoughlemma} of size at most $2$. Now in the case of $t = 1/2$, we show that there exist such vertex sets of size at most $1$. Unfortunately, this statement is not extendable to the case of $t = 3/2$; for an example see Figure~\ref{fig:clawfree_3/2_counterexample}.

\begin{figure}[ht]
\begin{center}
\begin{tikzpicture}[scale=1]
 \tikzstyle{vertex}=[draw,circle,fill=black,minimum size=4,inner sep=0]
 
 \node[vertex] (v1) at (90+0*72:0.75) {};
 \node[vertex] (v2) at (90+1*72:0.75) {};
 \node[vertex] (v3) at (90+2*72:0.75) {};
 \node[vertex] (v4) at (90+3*72:0.75) {};
 \node[vertex] (v5) at (90+4*72:0.75) {};
 
 \node[vertex] (u1) at (90+0*72:1.25) {};
 \node[vertex] (u2) at (90+1*72:1.25) {};
 \node[vertex] (u3) at (90+2*72:1.25) {};
 \node[vertex] (u4) at (90+3*72:1.25) {};
 \node[vertex] (u5) at (90+4*72:1.25) {};
 
 \node[vertex] (w1a) at (90-18+0*36:1.75) {};
 \node[vertex] (w1b) at (90-18+1*36:1.75) {};
 \node[vertex] (w2a) at (90-18+2*36:1.75) {};
 \node[vertex] (w2b) at (90-18+3*36:1.75) {};
 \node[vertex] (w3a) at (90-18+4*36:1.75) {};
 \node[vertex] (w3b) at (90-18+5*36:1.75) {};
 \node[vertex] (w4a) at (90-18+6*36:1.75) {};
 \node[vertex] (w4b) at (90-18+7*36:1.75) {};
 \node[vertex] (w5a) at (90-18+8*36:1.75) {};
 \node[vertex] (w5b) at (90-18+9*36:1.75) {};
 
 \draw[thick] (v4) -- (v5) -- (v1) -- (v2) -- (v3);
 \draw[thick] (v3) -- (v4) node [pos=0.5, shift={(0pt,-6pt)}] {$e$};
 \draw[thick] (v1) -- (v3) -- (v5) -- (v2) -- (v4) -- (v1);
 \draw[thick] (v1) -- (u1);
 \draw[thick] (v2) -- (u2);
 \draw[thick] (v3) -- (u3);
 \draw[thick] (v4) -- (u4);
 \draw[thick] (v5) -- (u5);
 \draw[thick] (w1a) -- (u1) -- (w1b);
 \draw[thick] (w2a) -- (u2) -- (w2b);
 \draw[thick] (w3a) -- (u3) -- (w3b);
 \draw[thick] (w4a) -- (u4) -- (w4b);
 \draw[thick] (w5a) -- (u5) -- (w5b);
 \draw[thick] (w1a) -- (w1b) -- (w2a) -- (w2b) -- (w3a) -- (w3b) -- (w4a) -- (w4b) -- (w5a) -- (w5b) -- (w1a);
\end{tikzpicture}
\caption{A minimally 3/2-tough, claw-free graph and its edge $e$ for which there exist no vertex sets $S=S(e)$ guaranteed by Proposition~\ref{prop:minttoughlemma} with $|S| \le 3$.}
\label{fig:clawfree_3/2_counterexample}
\end{center}
\end{figure}

\begin{lemma} \label{lemma:mintoughsets_in_min_1/2_tough_clawfree}
 If $G$ is a minimally $1/2$-tough, claw-free graph, then for every edge $e$ of $G$, there exists a vertex set $S = S(e) \subseteq V(G)$ guaranteed by Proposition~\ref{prop:minttoughlemma} with $|S| \le 1$.
\end{lemma}
\begin{proof}
 Let $e$ be an arbitrary edge of $G$ and $S=S(e)$ a vertex set guaranteed by Proposition~\ref{prop:minttoughlemma}. If $e$ is a bridge, then we can assume $S = \emptyset$, i.e.\ $|S| = 0$. So let us assume that $e$ is not a bridge. Then $S \ne \emptyset$ and $\omega(G-S) \le 2|S|$ and $\omega \big( (G-e)-S \big) > 2|S|$, which is only possible if $\omega(G-S) = 2|S|$ and $\omega \big( (G-e)-S \big) = 2|S| + 1$. Therefore, $S$ is a tough set, so by Lemma~\ref{lemma:structure_of_clawfree_graphs_with_tougness_t}, the vertices of $S$ have neighbors in exactly two components of $G-S$, and every component of $G-S$ has exactly one neighbor in $S$. Let $\{ v \} \subseteq V(G)$ be the neighborhood of the component of $G-S$ containing the edge $e$. Then $\{ v \}$ is a cutset, and its removal leaves exactly $2$ components, so it is a tough set having all the properties described in Proposition~\ref{prop:minttoughlemma}.
\end{proof}

This lemma has the following straightforward consequence.

\begin{corollary}
 The class of minimally $1/2$-tough, claw-free graphs can be recognized in polynomial time.
\end{corollary}

The following lemma describes another interesting property of minimally $1/2$-tough, claw-free graphs.

\begin{lemma} \label{lemma:cycles_in_min_1/2_tough_clawfree}
 If $G$ is a minimally $1/2$-tough, claw-free graph, then all of its cycles have length $3$.
\end{lemma}
\begin{proof}
 If every edge of $G$ is a bridge, then $G$ is a tree, so it does not contain any cycles.
 
 Let us assume that there exists an edge $e=uv$, which is not a bridge. Then by Lemma~\ref{lemma:mintoughsets_in_min_1/2_tough_clawfree}, there exists a vertex $w$ which is a cut-vertex in $G$ such that $e$ is a bridge in $G - \{ w \}$. Let $L_1$ and $L_2$ be the components of $G - \{ w \}$ (note that since $G$ is $1/2$-tough, $G - \{ w \}$ has exactly 2 components). Without loss of generality, we can assume $u, v \in L_1$. Let $L_{1,1}$ and $L_{1,2}$ denote the components of $(G-e) - \{ w \}$ containing $u$ and~$v$, respectively. Since $G$ is connected and $e$ is not a bridge in $G$, the vertex $w$ has neighbors in $L_{1,1}$, $L_{1,2}$, and~$L_2$. Since $G$ is claw-free and $w$ has neighbors in $L_2$, the neighbors of $w$ in $L_1$ must span a clique. Since $e$ is a bridge in $L_1$, and $w$ has neighbors both in $L_{1,1}$ and $L_{1,2}$, it follows that $w$ has exactly one neighbor in $L_{1,1}$, namely $u$, and has exactly one neighbor in $L_{1,2}$, namely $v$. Since $e$ is not a bridge in~$G$, there exists a cycle containing the edge $e$, but then this cycle must also contain the vertex $w$. By the previous observations, this cycle must be $\{ u,v,w \}$.
 
 This means that in $G$ every cycle has length $3$.
\end{proof}

Now we are ready to characterize minimally $1/2$-tough, claw-free graphs.

\begin{theorem} \label{thm:characterization_of_min_1/2_tough_clawfree}
The class of minimally $1/2$-tough, claw-free graphs consists of exactly those graphs that can be built up in the following way.
 \begin{enumerate}
  \item Take a tree $T$ with maximum degree at most 3 where the set of vertices of degree 1 and 3 together form an independent set.
  \item Now delete every vertex of degree 3, but connect its 3 neighbors with a triangle.
 \end{enumerate}
\end{theorem}
\begin{proof}
 Let $G$ be a graph that can be obtained as described in the theorem. It is easy to see that $G$ is claw-free. If $G$ does not contain a triangle, then it must be a path on at least 3 vertices, which is clearly minimally $1/2$-tough. If $G$ contains some triangles, then by the construction of $G$, when removing any set $S$ of vertices, the removal of each vertex of $S$ creates at most one more new component, so $G$ is $1/2$-tough. On the other hand, if an edge of a triangle is deleted, then by removing the third vertex of this triangle, we obtain 3 components since the vertices of the triangle were not leaves in the original tree $T$. All the other edges of the graph are bridges, so the graph is minimally $1/2$-tough.
 
 Now we show that if a graph $G$ is claw-free and minimally $1/2$-tough, then it can be obtained as described in the theorem.
 
 \emph{Case 1:} $G$ is a tree.
 
 Since $G$ is a claw-free tree, it cannot have a vertex of degree at least $3$, so $G$ must be a path. A path on $2$ vertices is $K_2$, whose toughness is infinity, so it is not minimally $1/2$-tough. A path on at least $3$ vertices is obtained by setting $T$ to be exactly this path. (In this case, Step 2 of the construction does not change~$T$.) 
 
 \emph{Case 2:} $G$ is not a tree.
 
 By Lemma~\ref{lemma:mintoughsets_in_min_1/2_tough_clawfree}, every vertex of every triangle is a cut-vertex. By Lemma~\ref{lemma:cycles_in_min_1/2_tough_clawfree}, two triangles cannot share an edge. Since $G$ is claw-free, any vertex not contained in any triangle has degree 1 or 2. By Lemma~\ref{lemma:cycles_in_min_1/2_tough_clawfree}, if a vertex is not contained in any triangle, then it is not contained in any cycle; thus it is either of degree~1 or is a cut-vertex. Now apply the reverse of the operation given in Step 2 (i.e., for each triangle, remove its edges, add a new vertex and connect it with the vertices of the triangle). Let us call the newly added vertices red, the vertices of the triangles green, and the other vertices blue. Since $G$ does not contain any cycles except for the triangles, the resulting graph must be a tree. Clearly, the red vertices have degree 3 and the blue ones have degree 1 or 2 in the tree. Now we show that the green vertices also have degrees 1 or 2 in the tree. Since the green vertices are contained in a triangle of $G$, they are cut-vertices. Since the toughness of $G$ is $1/2$, it is not difficult to see that the green vertices have degree 1 or 2 in the tree. So the red vertices have degree 3, and all the other vertices have degree 1 or 2 in the tree.

 To complete the proof, we need to show that the set of vertices of degrees 1 and 3 in the tree together form an independent set. Two leaves of the tree cannot be adjacent since the graph is connected and has at least 3 vertices. Two vertices of degree 3 of the tree (i.e.\ red vertices) cannot be adjacent because of the way they were created. Two vertices of degrees 1 and 3 of the tree cannot be adjacent since that would mean that this leaf of the tree is contained in a triangle in $G$ and has no neighbors outside the triangle, which contradicts the fact that it is a cut-vertex in $G$.
\end{proof}

For an example to obtain a minimally $1/2$-tough, claw-free graph described in Theorem~\ref{thm:characterization_of_min_1/2_tough_clawfree}, see Figure~\ref{fig:min_1/2_clawfree}.

\begin{figure}[ht]
\centering
\begin{tikzpicture}[scale=0.75]
 \tikzstyle{vertex}=[draw,circle,fill=black,minimum size=4,inner sep=0]
 
 \node[vertex] (o) at (0,0) {};
 
 \node[vertex] (a1) at (30:0.75) {};
 \node[vertex] (a2) at (150:0.75) {};
 \node[vertex] (a3) at (270:0.75) {};
 
 \node[vertex] (b1) at (30:1.5) {};
 \coordinate (b2) at (150:1.5) {};
 \node[vertex] (b3) at (270:1.5) {};
 
 \begin{scope}[shift={(b2)}]
 \node[vertex] (c1) at (90:0.75) {};
 \node[vertex] (c2) at (210:0.75) {};
 
 \node[vertex] (d1) at (90:1.5) {};
 \node[vertex] (d2) at (210:1.5) {};
 \end{scope}
 
 \coordinate (c3) at (30:2.25) {};
 
 \begin{scope}[shift={(c3)}]
 \node[vertex] (e1) at (90:0.75) {};
 \node[vertex] (e2) at (330:0.75) {};
 
 \node[vertex] (f1) at (90:1.5) {};
 \node[vertex] (f2) at (330:1.5) {};
 \end{scope}
 
 \draw[thick] (a1) -- (b1);
 \draw[thick] (a3) -- (b3);
 \draw[thick] (c1) -- (d1);
 \draw[thick] (c2) -- (d2);
 \draw[thick] (e1) -- (f1);
 \draw[thick] (e2) -- (f2);
 
 \draw[thick] (o) -- (a1);
 \draw[thick] (o) -- (a2);
 \draw[thick] (o) -- (a3);
 \node[vertex] at (150:1.5) {};
 \draw[thick] (b2) -- (a2);
 \draw[thick] (b2) -- (c1);
 \draw[thick] (b2) -- (c2);
 \node[vertex] at (30:2.25) {};
 \draw[thick] (c3) -- (b1);
 \draw[thick] (c3) -- (e1);
 \draw[thick] (c3) -- (e2);
 
 \node at (5,0) {$\longrightarrow$};
 
 \begin{scope}[shift={(10,0)}]
 \node[vertex] (a1) at (30:0.75) {};
 \node[vertex] (a2) at (150:0.75) {};
 \node[vertex] (a3) at (270:0.75) {};
 
 \node[vertex] (b1) at (30:1.5) {};
 \coordinate (b2) at (150:1.5);
 \node[vertex] (b3) at (270:1.5) {};
 
 \begin{scope}[shift={(b2)}]
 \node[vertex] (c1) at (90:0.75) {};
 \node[vertex] (c2) at (210:0.75) {};
 
 \node[vertex] (d1) at (90:1.5) {};
 \node[vertex] (d2) at (210:1.5) {};
 \end{scope}
 
 \coordinate (c3) at (30:2.25);
 
 \begin{scope}[shift={(c3)}]
 \node[vertex] (e1) at (90:0.75) {};
 \node[vertex] (e2) at (330:0.75) {};
 
 \node[vertex] (f1) at (90:1.5) {};
 \node[vertex] (f2) at (330:1.5) {};
 \end{scope}
 
 \draw[thick] (a1) -- (b1);
 \draw[thick] (a3) -- (b3);
 \draw[thick] (c1) -- (d1);
 \draw[thick] (c2) -- (d2);
 \draw[thick] (e1) -- (f1);
 \draw[thick] (e2) -- (f2);
 
 \draw[thick] (a1) -- (a2) -- (a3) -- (a1);
 \draw[thick] (a2) -- (c1) -- (c2) -- (a2);
 \draw[thick] (b1) -- (e1) -- (e2) -- (b1);
 \end{scope}
\end{tikzpicture}
\caption{Creating a minimally $1/2$-tough, claw-free graph from a tree according to Theorem~\ref{thm:characterization_of_min_1/2_tough_clawfree}.}
\label{fig:min_1/2_clawfree}
\end{figure}

With the help of this characterization, we can easily determine the minimum degree of minimally $1/2$-tough, claw-free graphs.

\begin{corollary} \label{cor:mindegree_of_min_1/2_tough_clawfree}
 Every minimally $1/2$-tough, claw-free graph has a vertex of degree $1$.
\end{corollary}
\begin{proof}
Theorem~\ref{thm:characterization_of_min_1/2_tough_clawfree} shows that every minimally $1/2$-tough, claw-free graph is obtained from a tree, which clearly contains a leaf. Also in the starting tree, the leaves cannot be adjacent to any vertex of degree 3, so the applied operation (i.e. Step 2) does not affect them.
\end{proof}

Since by Theorem~\ref{thm:toughness_of_clawfree}, the toughness of any claw-free graph is either an integer or half of an integer, Conjecture~\ref{conj:general_kriesell} is true in the class of claw-free graphs for all positive rational number $t \le 1$ by Corollary~\ref{cor:mindegree_of_min_1/2_tough_clawfree} and Theorem~\ref{thm:characterization_of_min_1_tough_clawfree}.

\section{\texorpdfstring{$2K_2$}{}-free graphs}

Finally, we study minimally tough, $2K_2$-free graphs.

\begin{definition}
 A graph is said to be $2K_2$-free if it does not contain an independent pair of edges as an induced subgraph.
\end{definition}

\begin{theorem}[\cite{article:2K2}] \label{article:2K2}
 For any positive rational number $t$, the class of $t$-tough, $2K_2$-free graphs can be recognized in polynomial time.
\end{theorem}

The following claim is obvious from the definitions.

\begin{claim} \label{claim:2K2_isolated_vertices}
 Let $G$ be a $2K_2$-free graph and $S \subseteq V(G)$ a cutset. Then in $G-S$, there is at most one component of size at least two.
\end{claim}

Now we prove that for every edge of a minimally $t$-tough, $2K_2$-free graph, there exists a vertex set guaranteed by Proposition~\ref{prop:minttoughlemma} with some nice properties.

\begin{lemma} \label{lemma:2K2_mintoughset}
 Let $t$ be a positive rational number, $G$ a minimally $t$-tough, $2K_2$-free graph and let $e = uv$ be an arbitrary edge of $G$. Then there exists a vertex set $S = S(e) \subseteq V(G)$ guaranteed by Proposition~\ref{prop:minttoughlemma} such that $|S|=1$ or $S$ is contained in the open neighborhood of $\{ u, v \}$, i.e.\ in the set of vertices adjacent to $u$ or $v$, excluding $u$ and $v$ themselves.
\end{lemma}
\begin{proof}
 Let $S$ be a vertex set guaranteed by Proposition~\ref{prop:minttoughlemma} and let us assume that there exists a vertex $w \in S$ such that $uw, vw \not\in E(G)$. Now we prove that $S=\{w\}$ or that there exists a vertex set $S'$ guaranteed by Proposition~\ref{prop:minttoughlemma} and a vertex $w' \notin S$ contained in the open neighborhood of $\{u,v\}$ for which $S' = \big( S \setminus \{w\} \big) \cup \{w'\}$.
 
 Let $L_u$ and $L_v$ denote the two components of $(G-e)-S$ for which $u \in L_u$ and $v \in L_v$. By Claim~\ref{claim:2K2_isolated_vertices}, all the components of $G-S$ are isolated vertices except for the component of the edge $e$. Since $G$ is $2K_2$-free, these isolated vertices cannot be adjacent to $w$. We can assume that $w$ has neighbors both in $L_u$ and $L_v$, otherwise we can consider $\widetilde{S} = S \setminus \{ w \}$ instead of $S$, see Figure~\ref{label}. Since $w$ is not adjacent to either $u$ or $v$, both $L_u$ and $L_v$ must have size at least two. Since $G$ is $2K_2$-free, both $L_u$ and $L_v$ are stars. (If there were any other edge in $L_u$, then this edge and any edge induced by $L_v$ would be independent.)
 
 \begin{figure}[ht]
 \begin{center}
 \begin{tikzpicture}[scale=1.25]
 \tikzstyle{vertex}=[draw,circle,fill=black,minimum size=4,inner sep=0]
 
 \node[vertex] at (-1.5,0) {};
 \node[vertex] at (-0.5,0) {};
 \node[vertex] (w) at (0.5,0) [label={[label distance=-3]315:$w$}] {};
 \node[vertex] at (1.5,0) {};
 
 \draw[rounded corners=5] (-1.75,-0.35) rectangle (1.75,0.35);
 \node at (2.1,-0.25) {$S$};
 
 \draw[thick] (1.75,1.5) ellipse (0.65 and 0.5);
 \draw[thick] (-0.25,1.5) ellipse (0.65 and 0.5);
 \node[vertex] (K1) at (-1.5,1.5) {};
 \node[vertex] (K2) at (-2,1.5) {};
 \node[vertex] (K3) at (-2.5,1.5) {};
 
 \node at (-0.75,2.25) {$L_u$};
 \node at (2.25,2.25) {$L_v$};
 
 \node[vertex] (u) at (0.25,1.5) [label={[label distance=-2.5]180: $u$}] {};
 \node[vertex] (v) at (1.25,1.5) [label={[label distance=-3]0:$v$}] {};
 
 \draw[thick] (u) to [bend left=25] node[pos=0.5, above] {$e$} (v);
 
 \draw[thick, dashed] (w) -- (K1);
 \draw[thick, dashed] (w) -- (K2);
 \draw[thick, dashed] (w) -- (K3);
 \draw[thick, dashed] (w) -- (u);
 \draw[thick, dashed] (w) -- (v);
 
 \node[vertex] (u1) at (-0.35,1.5) {};
 \node[vertex] (v1) at (1.85,1.5) {};
 
 \draw[thick] (w) -- (u1);
 \draw[thick] (w) -- (v1);
\end{tikzpicture}
\caption{The set $S = S(e)$ guaranteed by Proposition~\ref{prop:minttoughlemma} in a minimally $2K_2$-free graph with a vertex $w \in S$ which is not adjacent to any of the endpoints of $e$.}
\label{label}
\end{center}
\end{figure}
 
 Clearly, $S \cup \{ u \}$ is a cutset in $G$, and thus
 \[ \omega(G-S) + |L_u| - 1 = \omega \big( G - (S \cup \{u\}) \big) \le \frac{|S \cup \{ u \}|}{t} = \frac{|S|+1}{t} \text{.} \]
 Obviously, the same holds for $S \cup \{ v \}$. By Proposition~\ref{prop:minttoughlemma}, we obtain
 \[ \omega(G-S) = \omega((G-e)-S)-1 > \frac{|S|}{t} - 1 \text{.} \]
 Therefore,
 \[ \left( \frac{|S|}{t} - 1 \right) + \big( |L_u| - 1 \big) < \omega(G-S) + \big( |L_u| - 1 \big) \le \frac{|S|+1}{t} \]
 holds, i.e. $|L_u| < 2 + 1/t$. Similarly, we have $|L_v| < 2 + 1/t$.
 
 \medskip
 
 \emph{Case 1:} $S$ is a cutset in $G$.
 
 Then $S \setminus \{ w \}$ is also a cutset in $G$ (since $w$ has neighbors only in that component of $G-S$ which contains the edge $e$), so
 \[ \omega(G-S) = \omega \big( G - (S \setminus \{ w \}) \big) \le \frac{|S \setminus \{ w \}|}{t} = \frac{|S|-1}{t} \text{.} \]
 By Proposition~\ref{prop:minttoughlemma}, we obtain
 \[ \frac{|S|}{t} < \omega \big( (G-e) - S \big) = \omega(G-S) + 1 \le \frac{|S|-1}{t} + 1 \text{.} \]
 Therefore, $t > 1$ holds. So $|L_u| < 2+1/t < 3$, i.e.\ $|L_u| \le 2$. Hence, $L_u = \{ u, u_1 \}$ and similarly, $L_v = \{ v, v_1 \}$. So $L_u \cup L_v \cup \{ w \}$ spans a cycle of length $5$, namely $wu_1uvv_1w$. Therefore, $S' = \big( S \setminus \{ w \} \big) \cup \{ u_1 \}$ has the required properties.
 
 \medskip
 
 \emph{Case 2:} $S$ is not a cutset in $G$.
 
 \smallskip
 
 \emph{Case 2.1:} ($S$ is not a cutset in $G$ and) $t>1$.
 
 Then $|L_u| < 2 + 1/t <3$, so similarly as in Case 1, we obtain that $L_u = \{ u, u_1 \}$, $L_v = \{ v, v_1 \}$ and $S' = \big( S \setminus \{ w \} \big) \cup \{ u_1 \}$ has the required properties.
 
 \smallskip
 
 \emph{Case 2.2:} ($S$ is not a cutset in $G$ and) $t \le 1$.
 
 Then $|S|=1$ must hold, otherwise $S \setminus \{ w \} \ne \emptyset$, so by Proposition~\ref{prop:obs_below_1},
  \[ 1 = \omega(G-S) = \omega \Big( G- \big( S \setminus \{ w \} \big) \Big) \le \frac{\big| S \setminus \{ w \} \big|}{t} = \frac{|S|-1}{t} \text{,} \]
 and so by Proposition~\ref{prop:minttoughlemma},
  \[ \frac{|S|}{t} < \omega \big( (G-e)-S \big) = \omega(G-S) + 1 \le \frac{|S|-1}{t} + 1 \text{,} \]
 which would imply $t > 1$, a contradiction.
\end{proof}

\begin{theorem}
 For any positive rational number $t$, the class of minimally $t$-tough, $2K_2$-free graphs can be recognized in polynomial time.
\end{theorem}
\begin{proof}
 By Theorem~\ref{article:2K2}, we can compute the toughness of a $2K_2$-free graph in polynomial time. We only need to examine whether the removal of any edge from the graph decreases the toughness.
 
 Let $e = uv$ be an arbitrary edge of the graph. If $e$ is a bridge, then its removal from $G$ obviously decreases the toughness. Let us assume that $e$ is not a bridge. If $G$ is minimally $t$-tough, then by Proposition~\ref{prop:minttoughlemma}, there exists $S = S(e) \subseteq V(G)$ for which $\omega \big( (G-e) - S \big) > |S|/t$, and by Theorem~\ref{lemma:2K2_mintoughset}, we can assume that $|S|=1$ or that $S$ is contained in the open neighborhood of $\{ u,v \}$. First, check whether there exists a cut-vertex in $G-e$ whose removal from $G-e$ leaves more than $1/t$ components (this can be clearly done in polynomial time). If it does, then the removal of $e$ from $G$ obviously decreases the toughness. If it does not, then start a BFS algorithm in $G$ at $u$ and $v$ simultaneously (this can be also done in polynomial time). Since $G$ is $2K_2$-free, the BFS tree has at most two levels (not counting the zeroth level containing $u$ and $v$) and inside the second level there are no edges. If $G$ is minimally $t$-tough, then the above mentioned set $S$ exists and (since $|S| \ne 1$) it is contained in the open neighborhood of $\{ u, v \}$, i.e.\ in the first level of this BFS tree. In addition, by Claim~\ref{claim:2K2_isolated_vertices}, all the components of $G-S$ are isolated vertices except for the component of the edge $e$, thus every vertex in the first level belongs either to $S$ or to the component of $e$. Therefore, if we remove the edge $e$ from $G$ and expand the first level into a clique by adding all necessary edges, then the toughness of the obtained split graph is equal to the toughness of $G-e$. By Theorem~\ref{thm:split_toughness}, we can compute the toughness of this split graph in polynomial time and check if it is less than $t$.
 
 So we can decide in polynomial time whether a given $2K_2$-free graph is minimally tough.
\end{proof}

Finally, we give a few examples of minimally tough, $2K_2$-free graphs:
\begin{itemize}
 \item[--] the graphs $C_4$ and $C_5$ are minimally $1$-tough, $2K_2$-free graphs,
 \item[--] for any integer $b \ge 2$, the graph $K_{1,b}$ is a minimally $1/b$-tough, $2K_2$-free graph,
 \item[--] the path on 4 vertices is a minimally $1/2$-tough, $2K_2$-free graph,
 \item[--] the two graphs in Figure~\ref{fig:2K2free_2/3_example} are minimally $2/3$-tough, $2K_2$-free graphs.
\end{itemize}

\begin{figure}[ht] 
\begin{center}
\begin{tikzpicture}
 \tikzstyle{vertex}=[draw,circle,fill=black,minimum size=4,inner sep=0]
 
 \node[vertex] (1) at (90:1) {};
 \node[vertex] (2) at (90+72:0.75) {};
 \node[vertex] (3) at (90+2*72:1) {};
 \node[vertex] (4) at (90+3*72:1) {};
 \node[vertex] (5) at (90+4*72:0.75) {};
 \node[vertex] (2a) at (90+72:1.25) {};
 \node[vertex] (5a) at (90+4*72:1.25) {};
 
 \draw[thick] (1) -- (2) -- (3) -- (4) -- (5) -- (1);
 \draw[thick] (1) -- (2a) -- (3);
 \draw[thick] (1) -- (5a) -- (4);
 
 \begin{scope}[shift={(5,0)}]
  \node[vertex] (1) at (90:1) {};
 \node[vertex] (2) at (90+72:1) {};
 \node[vertex] (3) at (90+2*72:1) {};
 \node[vertex] (4) at (90+3*72:1) {};
 \node[vertex] (5) at (90+4*72:0.75) {};
 \node[vertex] (5a) at (90+4*72:1.25) {};
 
 \draw[thick] (1) -- (2) -- (3) -- (4) -- (5) -- (1);
 \draw[thick] (1) -- (5a) -- (4);
 \end{scope}
\end{tikzpicture}
\caption{Two minimally $2/3$-tough, $2K_2$-free graphs.}
\label{fig:2K2free_2/3_example}
\end{center}
\end{figure}

Conjecture~\ref{conj:general_kriesell} remains open for minimally tough, $2K_2$-free graphs.

\acknowledgements
We would like to thank Binlong Li for the construction in Figure~\ref{fig:clawfree_3/2_counterexample}.

\nocite{*}
\bibliographystyle{abbrvnat}
\bibliography{special_graph_classes}
\label{sec:biblio}

\end{document}